\documentclass {myjournal}      
\allowdisplaybreaks

\title{Analytical evaluation of some non-elementary  integrals involving some exponential, hyperbolic and trigonometric elementary functions and derivation of new probability measures generalizing the gamma-type and normal distributions}
\headlinetitle{Some non-elementary integrals and applications in probability theory} 
\lastnameone{Nijimbere}
\firstnameone{Victor}
\nameshortone{V.~Nijimbere}
\addressone{School of Mathematics and Statistics, Carleton University, Ottawa, Ontario}
\countryone{Canada}
\emailone{victornijimbere@gmail.com}

\lastnametwo{}
\firstnametwo{}
\nameshorttwo{}
\addresstwo{}
\countrytwo{}
\emailtwo{}

\lastnamethree{}
\firstnamethree{}
\nameshortthree{}
\addressthree{}
\countrythree{}
\emailthree{}

\lastnamefour{}
\firstnamefour{}
\nameshortfour{}
\addressfour{}
\countryfour{}
\emailfour{}

\lastnamefive{}
\firstnamefive{}
\nameshortfive{}
\addressfive{}
\countryfive{}
\emailfive{}

\abstract{The non-elementary  integrals involving elementary exponential, hyperbolic and trigonometric functions,
$
\int x^\alpha e^{\eta x^\beta}dx, \int x^\alpha \cosh\left(\eta x^\beta\right)dx, \int x^\alpha \sinh\left(\eta x^\beta\right)dx, \int x^\alpha \cos\left(\eta x^\beta\right)dx$ and $\int x^\alpha \sin\left(\eta x^\beta\right)dx $
where $\alpha, \eta$ and $\beta$ are real or complex constants are evaluated in terms of the confluent hypergeometric function $_1F_1$ and  the  hypergeometric function $_1F_2$.  The hyperbolic and Euler identities are used to derive some identities involving exponential, hyperbolic, trigonometric functions and the hypergeometric functions $_1F_1$ and  $_1F_2$. Having evaluated, these non-elementary integrals, some new probability measures generalizing the gamma-type and Gaussian distributions are also obtained. The obtained generalized distributions may, for example, allow to perform better statistical tests than those already known (e.g. chi-square ($\chi^2$) statistical tests and other statistical tests constructed based on the central limit theorem (CLT)).}



\begin{document}

\section{Introduction}
\label{sec:1}

The confluent hypergeometric function $_1F_1$ and  the  hypergeometric function $_1F_2$ are used throughout this paper. There are defined here for reference.
\begin{definition}The confluent  hypergeometric function, denoted as $_1F_1$, is a special function given by the series \cite{AB,NI}
\begin{equation}
_1 F_1(a;b; x)=\sum\limits_{n=0}^{\infty}\frac{(a)_n}{(b)_n}\frac{x^n}{n!},
\label{eq1.1}
\end{equation} 
where $a$ and $b$ are arbitrary constants, $(\vartheta)_n=\Gamma(\vartheta+n)/\Gamma(\vartheta)$ (Pochhammer's notation \cite{AB}) for any complex $\vartheta$, with $(\vartheta)_0=1$, and $\Gamma$ is the standard gamma function \cite{AB,NI}.
\label{def2}
\end{definition}

\begin{definition}The hypergeometric function $_1F_2$, is a special function given by the series \cite{AB,NI}
\begin{equation}
_1F_2(a;b,c; x)=\sum\limits_{n=0}^{\infty}\frac{(a)_n}{(b)_n (c)_n}\frac{x^n}{n!},
\label{eq1.2}
\end{equation} 
where $a, b$ and $c$ are arbitrary constants, $(\vartheta)_n=\Gamma(\vartheta+n)/\Gamma(\vartheta)$ (see definition \ref{def2})\cite{AB,NI}.
\label{def2}
\end{definition}

\begin{definition}
An elementary function is a function of one variable constructed using that variable and constants, and by performing a finite number of repeated algebraic operations involving exponentials and logarithms. An indefinite integral which can be expressed in terms of elementary functions is an elementary integral. And if, on the other hand, it cannot be evaluated in terms of
elementary functions, then it is non-elementary \cite{M,R}.
\label{def1}
\end{definition}
 
One of the goals of this work is to show how non-elementary integrals having one of the types
\begin{equation}
\int x^\alpha e^{\eta x^\beta}dx, \int x^\alpha \cosh\left(\eta x^\beta\right)dx, \int x^\alpha \sinh\left(\eta x^\beta\right)dx,
\label{eq1.3}
\end{equation}
\begin{equation}
\int x^\alpha \cos\left(\eta x^\beta\right)dx\,\,\,\mbox{and}\,\,\,\int x^\alpha \sin\left(\eta x^\beta\right)dx,
\label{eq1.4}
\end{equation}
where $\alpha, \eta$ and $\beta$ are real or complex constants can be evaluated in terms of the special functions $_1F_1$ and $_1F_2$.  These integrals are the generalization of the non-elementary integrals evaluated  by Nijimbere \cite{N1,N2,N3}, and  have been not evaluated before. For instance, if $\alpha<0$, the integrals in (\ref{eq1.4}) become, respectively, the (indefinite) sine and cosine integrals which are evaluated in Nijimbere \cite{N2,N3}. If, on the other hand, $\alpha=0$, the non-elementary integrals in (\ref{eq1.3}) and  (\ref{eq1.4}) reduce to the non-elementary integrals evaluated in Nijimbere \cite{N1}.

However, it is important to observe that the integrals in (\ref{eq1.3}) and  (\ref{eq1.4}) may be elementary or non-elementary depending on the values of the constants $\alpha$ and $\beta$. If, for instance, $\alpha=\beta-1$, then the integral 
\begin{equation}
\int x^\alpha e^{\eta x^\beta}dx=\frac{1}{ \eta\beta}\int \eta\beta x^{\beta-1} e^{\eta x^\beta}dx= \frac{ e^{\eta x^\beta}}{ \eta\beta}+C
\label{eq1.5}
\end{equation}
is elementary because it is expressed in terms of the elementary function $ e^{\eta x^\beta}$. In that case, the other integrals in (\ref{eq1.3}) and  (\ref{eq1.4}) are also elementary since they can be expressed as linear combination of  integrals such that in (\ref{eq1.5}) using the hyperbolic identities 
$$ \cosh\left(\eta x^\beta\right)= \left( e^{\eta x^\beta}+ e^{-\eta x^\beta}\right)/2,\,\, \sinh\left(\eta x^\beta\right)= \left( e^{\eta x^\beta}-e^{-\eta x^\beta}\right)/2$$
and the Euler's identities 
$$ \cos\left(\eta x^\beta\right)= \left( e^{i\eta x^\beta}+ e^{-i\eta x^\beta}\right)/2,\,\, \sin\left(\eta x^\beta\right)= \left( e^{i\eta x^\beta}-e^{-i\eta x^\beta}\right)/(2i).$$
Using Liouville 1835's theorem, it can readily be shown that if $\alpha$ is not an integer and $\alpha\ne\beta-1$, then the integrals in (\ref{eq1.4}) and (\ref{eq1.5}) are non-elementary \cite{M,R}.
Another goal of this work is to obtain some identities (or formula) involving exponential, hyperbolic, trigonometric functions and the hypergeometric functions $_1F_1$ and  $_1F_2$ using the Euler and  hyperbolic identities. Other interesting identities involving hypergeometric functions may be found, for example, in \cite{N1,N2,N3,AR,CR,QQ}.  Non-elementary Integrals with integrands involving generalized hypergeometric functions and identities of generalized hypergeometric series have also been examined  in Nijimbere \cite{ N5}.

It is well known that numerical integrations (or approximations) are expensive and their main drawback is that they are associated with computational errors which become very large as the integration limits become large. Thus, the analytical method used in this paper (developed by  Nijimbere \cite{N3}) is very important in order to avoid computational methods, see for example the case of Dawson's integral and related functions  in mathematical physics \cite{N4}.

 Using the fact that $g(x)= e^{-\eta x^\beta},\,x\in\mathbb{R},\,\eta\in\mathbb{R}^+$, is in the $L^p$-space, $p>0$ for some $\beta\in\mathbb{R}$, some finite measure, $\mu(\{-\infty,x\})<\infty$, can be defined for all $x\in \mathbb{R}$. Moreover, if $X=h(x), x\in\mathbb{R}$ is some random variable, $h:\mathbb{R}\to\mathbb{R}$ is some well-defined function (e.g. $h(x)=x$), then  it is possible to define probability measures in terms of the Lebesgue measure $dx$ as $\mu(dx)=A\,g(x)dx, x\in\Omega$, and $\Omega\subseteq \mathbb{R}$, satisfying the integrability condition $\int_{\Omega\subseteq \mathbb{R}} |X|^{ \alpha} \mu(dx)<\infty, \alpha\ne0, \alpha>-\beta-1$ and $A$ being a (normalization) constant. In that case, probability measures (or distributions) that generalize the gamma-type and Gaussian-type distributions may be constructed, and corresponding distribution functions and moments can be evaluated as well. For example, it can be shown using the  results in this paper that the $\mbox{n}^{th}$  moments of the Gaussian random variable  are given by the formula
$$M(X^n)
= \frac{\theta^n}{\sqrt{\pi}}\sum\limits_{l=0}^{n}\Gamma \left(l+1/2\right)C_{2l}^n\left(\frac{2\sigma^2}{\theta^2}\right)^{l}, \, 2l\le n\, \mbox{and}\,(2l)\in\mathbb{N},$$
where $\theta\in \mathbb{R}$ is the mean of the Gaussian random variable and $\sigma^2>0$ its variance. It is also shown, for instance, that the inverse gamma distribution, frequently used in signal processing in wireless communications, see for instance \cite{C,W}, is as well a particular case of the generalized gamma-type distribution derived in this study.

The paper is organized as follows. In section 2, the integrals in (\ref{eq1.4})-(\ref{eq1.5}) are evaluated, and some identities (or formula) that involve the exponential, hyperbolic, trigonometric functions and the hypergeometric functions $_1F_1$ and  $_1F_2$ are obtained. In section 3, probability measures that generalize the gamma-type and Gaussian-type distributions are constructed, and their corresponding distribution functions are written in terms of the confluent hypergeometric function. Formulas to evaluate the $\mbox{n}^{th}$ moments are also derived in section 3. A general discussion is given in section 4.

\section{Evaluation of the non-elementary integrals}
\label{sec:2} 
Let first prove an important lemma which will be used throughout the paper.
\begin{lemma}
Let $ j\ge0$ and $m\ge0$ be integers, and let $\alpha, \beta$ and $\gamma$ be arbitrarily constants. 
\begin{enumerate}
\item Then \begin{equation}
\prod \limits_{m=0}^{j}(\alpha+m\beta+1)=(\alpha+1)\beta^j\left(\frac{\alpha+1}{\beta}+1\right)_j,
\label{eq2.1}
\end{equation}
\item \begin{multline}
\prod \limits_{m=0}^{2j}(\alpha+m\beta+1)\\=(\alpha+1)(\alpha+\beta+1)(2\beta)^{2j}\left(\frac{\alpha+\beta+1}{2\beta}+1\right)_j\left(\frac{\alpha+2\beta+1}{2\beta}+1\right)_j
\label{eq2.2}
\end{multline}
\item and 
\begin{multline}
\prod \limits_{m=0}^{2j+1}(\alpha+m\beta+1)\\=(\alpha+1)(\alpha+\beta+1)(2\beta)^{2j}\left(\frac{\alpha+2\beta+1}{2\beta}+1\right)_j\left(\frac{\alpha+3\beta+1}{2\beta}+1\right)_j. 
\label{eq2.3}
\end{multline}
\end{enumerate}
\label{lem1}
\end{lemma}
\begin{proof} 
\begin{enumerate} 
\item Making use of  Pochhammer's notation \cite{AB,NI}, see definition \ref{def1} yields
\begin{align}
\prod \limits_{m=0}^{j}(\alpha+m\beta+1)&=(\alpha+1)\prod \limits_{m=1}^{j}\left(\alpha+m\beta+1\right)
\nonumber\\ &= (\alpha+1)\beta^j \prod \limits_{m=1}^{j}\left(\frac{\alpha+1}{\beta}+m\right)
\nonumber\\ &=  (\alpha+1)\beta^j \prod \limits_{m=1}^{j}\left(\frac{\alpha+1}{\beta}+1+m-1\right)
\nonumber\\ &=(\alpha+1)\beta^j\left(\frac{\alpha+1}{\beta}+1\right)_j.
\label{eq2.4}
\end{align}
\item Observe that  \begin{align} 
\prod \limits_{m=0}^{2j}(\alpha+m\beta+1)&=\prod \limits_{l=0}^{j-1}(\alpha+l(2\beta)+\beta+1)\prod \limits_{l=0}^{j}(\alpha+l(2\beta)+1).
\label{eq2.5}
\end{align}
Then,  making use of  Pochhammer's notation as before gives
\begin{align}
\prod \limits_{l=0}^{j-1}(\alpha+l(2\beta)+\beta+1)=(\alpha+\beta+1)(2\beta)^j\left(\frac{\alpha+\beta+1}{2\beta}+1\right)_j
\label{eq2.6}
\end{align}
and
\begin{align}
\prod \limits_{l=0}^{j}(\alpha+l(2\beta)+1)=(\alpha+1)(2\beta)^j\left(\frac{\alpha+2\beta+1}{2\beta}+1\right)_j.
\label{eq2.7}
\end{align} 
Hence, multiplying (\ref{eq2.6})  with  (\ref{eq2.7}) gives (\ref{eq2.2}). 
\item Observe that  
\begin{align} 
\prod \limits_{m=0}^{2j+1}(\alpha+m\beta+1)&=\prod \limits_{l=0}^{j}(\alpha+l(2\beta)+1)\prod \limits_{l=0}^{j}(\alpha+l(2\beta)+\beta+1).
\label{eq2.8}
\end{align}
Once again, using again Pochhammer's notation yields
\begin{align}
\prod \limits_{l=0}^{j}(\alpha+l(2\beta)+\beta+1)=(\alpha+\beta+1)(2\beta)^j\left(\frac{\alpha+3\beta+1}{2\beta}+1\right)_j.
\label{eq2.9}
\end{align}
Hence, multiplying (\ref{eq2.9}) with  (\ref{eq2.7}) gives (\ref{eq2.3}). 
\end{enumerate}
\end{proof}

\subsection{Evaluation of non-elementary integrals of the types $\int x^\alpha e^{\eta x^\beta}dx, \int x^\alpha \cosh\left(\eta x^\beta\right)dx, \int x^\alpha \sinh\left(\eta x^\beta\right)dx$}
\label{sec:2.1} 
\begin{proposition} Let  $\eta$ and $\beta$ be nonzero constants ($\eta\ne 0, \beta\ne0$), and  $\alpha$ be any constant different from $-1$ ($\alpha\ne-1$). Then,
\begin{equation}
\int x^\alpha e^{\eta x^\beta}dx=\frac{x^{\alpha+1}e^{\eta x^\beta}}{\alpha+1}\, _{1}F_{1} \Bigl(1;\frac{\alpha+\beta+1}{\beta}; -\eta x^{\beta}\Bigr)+C. 
\label{eq2.10}
\end{equation}
\label{prp1}
\end{proposition}

\begin{proof} 
The substitution $u^\beta=\eta x^\beta$ and (\ref{eq1.1}) yields
\begin{equation}
\int x^\alpha e^{\eta x^\beta}dx=
\frac{1}{\eta^{\frac{\alpha+1}{\beta}}}\int u^{\alpha} e^{u^\beta}du.
\label{eq2.11}
\end{equation}
Performing successive integration by parts that increases the power of $u$ gives
\begin{align}
\int u^{\alpha} e^{u^\beta}du =&\frac{u^{\alpha+1}e^{u^\beta}}{\alpha+\beta+1}-\frac{\beta u^{\alpha+\beta+1}e^{u^\beta}}{(\alpha+1)(\alpha+\beta+1)} \nonumber\\ &+\frac{\beta^2 u^{\alpha+2\beta+1}e^{u^\beta}}{(\alpha+1)(\alpha+\beta+1)(\alpha+2\beta+1)}\nonumber\\ &-\frac{\beta^3  u^{\alpha+3\beta+1}e^{u^\beta}}{(\alpha+1)(\alpha+\beta+1)(\alpha+2\beta+1)(\alpha+3\beta+1)}
\nonumber\\ &+\cdot\cdot\cdot\cdot+\frac{(-1)^j\beta^j  u^{\alpha+j\beta+1}e^{u^\beta}}{\prod\limits_{m=0}^{j}(\alpha+m\beta+1)}+\cdot\cdot\cdot\cdot
\nonumber\\ &=\sum\limits_{j=0}^{\infty}\frac{(-1)^j\beta^j  u^{\alpha+j\beta+1}e^{u^\beta}}{\prod\limits_{m=0}^{j}(\alpha+m\beta+1)}+C.
\label{eq2.12}
\end{align}
Using (\ref{eq2.1}) in Lemma \ref{lem1} yields
\begin{align}
\int u^{\alpha} e^{u^\beta}du & = u^{\alpha+1}e^{u^\beta}\sum\limits_{j=0}^{\infty}\frac{(-\beta  u^{\beta})^j}{\prod\limits_{m=0}^{j}(\alpha+m\beta+1)}+C
\nonumber\\ & =  u^{\alpha+1}e^{u^\beta}\sum\limits_{j=0}^{\infty}\frac{(-\beta  u^{\beta})^j}{ (\alpha+1)\beta^j\left(\frac{\alpha+1}{\beta}+1\right)_j} +C
\nonumber\\ & = \frac{ u^{\alpha+1}e^{u^\beta}}{\alpha+1}\sum\limits_{j=0}^{\infty}\frac{(- u^{\beta})^j}{ \left(\frac{\alpha+1}{\beta}+1\right)_j} +C= \frac{ u^{\alpha+1}e^{u^\beta}}{\alpha+1}\sum\limits_{j=0}^{\infty}\frac{(1)_j (- u^{\beta})^j}{ \left(\frac{\alpha+1}{\beta}+1\right)_j j!} +C
\nonumber\\ &= \frac{ u^{\alpha+1}e^{u^\beta}}{\alpha+1}\, _{1}F_{1} \Bigl(1;\frac{\alpha+\beta+1}{\beta}; -u^{\beta}\Bigr)+C.
\label{eq2.13}
\end{align}
Hence, using the fact $u^\beta=\eta x^\beta$ gives (\ref{eq2.2}).
\end{proof}

Having evaluated (\ref{eq2.10}), the following results hold.

\begin{theorem}
Let $\alpha$  be an arbitrarily real or complex constant, $\beta$ a nonzero real or complex constant ($\beta\ne0$), and $\eta$  a nonzero real or complex constant with a positive real part ($\mbox{Re}(\eta)>0$).

\begin{enumerate}

\item Then,
\begin{equation}
\int\limits_{0}^{+\infty} x^\alpha e^{-\eta x^\beta}dx=\frac{\Gamma \left(\frac{\alpha+\beta+1}{\beta}\right)}{(\alpha+1)\,\eta^{\frac{\alpha+1}{\beta}}},
\label{eq2.14}
\end{equation}
 $\alpha>-\beta-1, \alpha\ne-1$ if $\{\alpha,\beta\}\in \mathbb{R}$. 
\item Moreover, if the integrand is even, then 
\begin{equation}
\int\limits_{-\infty}^{+\infty} x^\alpha e^{-\eta x^\beta}dx=\frac{2\Gamma \left(\frac{\alpha+\beta+1}{\beta}\right)}{(\alpha+1)\,\eta^{\frac{\alpha+1}{\beta}}}.
\label{eq2.15}
\end{equation}
\end{enumerate}
\label{th1}
\end{theorem}

\begin{proof}
It can readily be shown using Proposition \ref{prp1} and the asymptotic expansion of the confluent hypergeometric function (formula 13.1.5 in \cite{AB}) that 
\begin{equation}
\int\limits_{0}^{+\infty} x^\alpha e^{-\eta x^\beta}dx=\lim\limits_{x\to\infty}\frac{x^{\alpha+1}e^{-\eta x^\beta}}{\alpha+1}\, _{1}F_{1} \Bigl(1;\frac{\alpha+\beta+1}{\beta}; \eta x^{\beta}\Bigr)= \frac{\Gamma \left(\frac{\alpha+\beta+1}{\beta}\right)}{(\alpha+1)\,\eta^{\frac{\alpha+1}{\beta}}}. 
\label{eq2.16}
\end{equation}
If the integrand is even, then $\int_{-\infty}^{+\infty} x^\alpha e^{-\eta x^\beta}dx=2\int_{0}^{+\infty} x^\alpha e^{-\eta x^\beta}dx$, and this gives (\ref{eq2.15}).
\end{proof}

 Theorem \ref{th1} is, for instance, the generalization of the Mellin transform of the function $e^{-\eta x^{\beta}},\,\text{Re}\{\eta\}>0, \beta>0$, where $s=\alpha+1$ is the Mellin parameter, and in this case it can be negative $s=\alpha+1<0$, and the constant $\beta$ can be negative as well ($\beta<0$), see for example Polarikas \cite{P}.

As it will shortly be shown (see section \ref{sec:3}), Theorem \ref{th1} can be used to obtain new probability distributions that generalize the gamma-type and Gaussian-type distributions that may lead to better statistical tests than those already known which are based on the central limit theorem (CLT) \cite{B}.

\begin{proposition}  Let  $\eta$ and $\beta$ be nonzero constants ($\eta\ne 0, \beta\ne0$), $\alpha$ be some constant different from $-1$ ($\alpha\ne-1$) and $\alpha\ne-\beta-1$. Then,
\begin{multline}
\int x^\alpha \cosh\left(\eta x^\beta\right)dx\\=\frac{x^{\alpha+1}}{(\alpha+1)(\alpha+\beta+1)}\Bigl[\cosh\left(\eta x^\beta\right)\,_{1}F_{2} \Bigl(1; \frac{\alpha+\beta+1}{2\beta}, \frac{\alpha+2\beta+1}{2\beta}; \frac{\eta^2 x^{2\beta}}{4}\Bigr)\\- \beta\eta x^{\beta}\sinh\left(\eta x^\beta\right)\,_{1}F_{2} \Bigl(1; \frac{\alpha+2\beta+1}{2\beta}, \frac{\alpha+3\beta+1}{2\beta}; \frac{\eta^2 x^{2\beta}}{4}\Bigr)\Bigr]+C. 
\label{eq2.17} 
\end{multline}
\label{prp2}
\end{proposition}

\begin{proof} 
The change of variable $u^\beta=\eta x^\beta$ yields
\begin{equation}
\int x^\alpha \cosh\left(\eta x^\beta\right)dx=
\frac{1}{\eta^{\alpha+1}{\beta}}\int u^{\alpha} \cosh\left(u^\beta\right)du+C.
\label{eq2.18}
\end{equation}
Successive integration by parts that increases the power of $u$ gives
\begin{align}
&\int u^{\alpha} \cosh\left(u^\beta\right)du =\frac{u^{\alpha+1}\cosh\left(u^\beta\right)}{\alpha+\beta+1}-\frac{\beta u^{\alpha+\beta+1}\sinh\left(u^\beta\right)}{(\alpha+1)(\alpha+\beta+1)} \nonumber\\ &+\frac{\beta^2 u^{\alpha+2\beta+1}\cosh\left(u^\beta\right)}{(\alpha+1)(\alpha+\beta+1)(\alpha+2\beta+1)}\nonumber\\ &-\frac{\beta^3  u^{\alpha+3\beta+1}\sinh\left(u^\beta\right)}{(\alpha+1)(\alpha+\beta+1)(\alpha+2\beta+1)(\alpha+3\beta+1)}
\nonumber\\ &+\cdot\cdot+\frac{\beta^{2j}  u^{\alpha+2j\beta+1}\cosh\left(u^\beta\right)}{\prod\limits_{m=0}^{2j}(\alpha+m\beta+1)}+\cdots-\frac{\beta^{2j+1}  u^{\alpha+(2j+1)\beta+1}\sinh\left(u^\beta\right)}{\prod\limits_{m=0}^{2j+1}(\alpha+m\beta+1)}-\cdot\cdot
\nonumber\\ &=\cosh\left(u^\beta\right)\sum\limits_{j=0}^{\infty}\frac{\beta^{2j}  u^{\alpha+2j\beta+1}}{\prod\limits_{m=0}^{2j}(\alpha+m\beta+1)}-\sinh\left(u^\beta\right)\sum\limits_{j=0}^{\infty}\frac{\beta^{2j+1}  u^{\alpha+(2j+1)\beta+1}}{\prod\limits_{m=0}^{2j+1}(\alpha+m\beta+1)}+C.
\label{eq2.19}
\end{align}
Using (\ref{eq2.2}) and  (\ref{eq2.3}) in Lemma \ref{lem1} yields
\begin{multline}
\int u^{\alpha}\cosh\left(u^\beta\right)du =\frac{ u^{\alpha+1}\cosh\left(u^\beta\right)}{(\alpha+1)(\alpha+\beta+1)}\sum\limits_{j=0}^{\infty}\frac{\left(\frac{ u^{2\beta}}{4}\right)^{j}}{\left(\frac{\alpha+\beta+1}{2\beta}\right)_j \left(\frac{\alpha+2\beta+1}{2\beta}\right)_j}\\ -\frac{\beta u^{\alpha+\beta+1}\sinh\left(u^\beta\right)}{(\alpha+1)(\alpha+\beta+1)}\sum\limits_{j=0}^{\infty}\frac{\left(\frac{ u^{2\beta}}{4}\right)^{j}}{\left(\frac{\alpha+2\beta+1}{2\beta}\right)_j \left(\frac{\alpha+3\beta+1}{2\beta}\right)_j}+C\\
=\frac{ u^{\alpha+1}\cosh\left(u^\beta\right)}{(\alpha+1)(\alpha+\beta+1)}\sum\limits_{j=0}^{\infty}\frac{(1)_j\left(\frac{ u^{2\beta}}{4}\right)^{j}}{\left(\frac{\alpha+\beta+1}{2\beta}\right)_j \left(\frac{\alpha+2\beta+1}{2\beta}\right)_j j!}\\ -\frac{ \beta u^{\alpha+\beta+1}\sinh\left(u^\beta\right)}{(\alpha+1)(\alpha+\beta+1)}\sum\limits_{j=0}^{\infty}\frac{(1)_j\left(\frac{ u^{2\beta}}{4}\right)^{j}}{\left(\frac{\alpha+2\beta+1}{2\beta}\right)_j \left(\frac{\alpha+3\beta+1}{2\beta}\right)_j j!}+C\\
=\frac{ u^{\alpha+1}\cosh\left(u^\beta\right)}{(\alpha+1)(\alpha+\beta+1)}\,_{1}F_{2} \Bigl(1; \frac{\alpha+\beta+1}{2\beta}, \frac{\alpha+2\beta+1}{2\beta}; \frac{u^{2\beta}}{4}\Bigr)\\ -\frac{ \beta u^{\alpha+\beta+1}\sinh\left(u^\beta\right)}{(\alpha+1)(\alpha+\beta+1)}\,_{1}F_{2} \Bigl(1; \frac{\alpha+2\beta+1}{2\beta}, \frac{\alpha+3\beta+1}{2\beta}; \frac{u^{2\beta}}{4}\Bigr)+C.
\label{eq2.20}
\end{multline}
Hence, using the fact $u^\beta=\eta x^\beta$ and rearranging terms gives (\ref{eq2.17}).
\end{proof}

\begin{proposition} Let  $\eta$ and $\beta$ be nonzero constants ($\eta\ne 0, \beta\ne0$), $\alpha$ be some constant different from $-1$ ($\alpha\ne-1$) and $\alpha\ne-\beta-1$. Then,
\begin{multline}
\int x^\alpha \sinh\left(\eta x^\beta\right)dx\\=\frac{x^{\alpha+1}}{(\alpha+1)(\alpha+\beta+1)}\Bigl[\sinh\left(\eta x^\beta\right)\,_{1}F_{2} \Bigl(1; \frac{\alpha+\beta+1}{2\beta}, \frac{\alpha+2\beta+1}{2\beta}; \frac{\eta^2 x^{2\beta}}{4}\Bigr)\\- \beta\eta x^{\beta}\cosh\left(\eta x^\beta\right)\,_{1}F_{2} \Bigl(1; \frac{\alpha+2\beta+1}{2\beta}, \frac{\alpha+3\beta+1}{2\beta}; \frac{\eta^2 x^{2\beta}}{4}\Bigr)\Bigr]+C. 
\label{eq2.21} 
\end{multline}
\label{prp3}
\end{proposition}

\begin{proof} 
Making the change of variable $u^\beta=\eta x^\beta$ as before yields
\begin{equation}
\int x^\alpha \sinh\left(\eta x^\beta\right)dx=
\frac{1}{\eta^{\frac{\alpha+1}{\beta}}}\int u^{\alpha} \sinh\left(u^\beta\right)du+C.
\label{eq2.22}
\end{equation}
Performing successive integration by parts that increase the power of $u$ as before gives
\begin{multline}
\int u^{\alpha} \sinh\left(u^\beta\right)du =\frac{u^{\alpha+1}\sinh\left(u^\beta\right)}{\alpha+\beta+1}-\frac{\beta u^{\alpha+\beta+1}\cosh\left(u^\beta\right)}{(\alpha+1)(\alpha+\beta+1)}\\+\frac{\beta^2 u^{\alpha+2\beta+1}\sinh\left(u^\beta\right)}{(\alpha+1)(\alpha+\beta+1)(\alpha+2\beta+1)}\\ -\frac{\beta^3  u^{\alpha+3\beta+1}\cosh\left(u^\beta\right)}{(\alpha+1)(\alpha+\beta+1)(\alpha+2\beta+1)(\alpha+3\beta+1)}
\\ +\cdot\cdot+\frac{\beta^{2j}  u^{\alpha+2j\beta+1}\sinh\left(u^\beta\right)}{\prod\limits_{m=0}^{2j}(\alpha+m\beta+1)}+\cdot\cdot-\frac{\beta^{2j+1}  u^{\alpha+(2j+1)\beta+1}\cosh\left(u^\beta\right)}{\prod\limits_{m=0}^{2j+1}(\alpha+m\beta+1)}-\cdot\cdot\\ =\sinh\left(u^\beta\right)\sum\limits_{j=0}^{\infty}\frac{\beta^{2j}  u^{\alpha+2j\beta+1}}{\prod\limits_{m=0}^{2j}(\alpha+m\beta+1)}-\cosh\left(u^\beta\right)\sum\limits_{j=0}^{\infty}\frac{\beta^{2j+1}  u^{\alpha+(2j+1)\beta+1}}{\prod\limits_{m=0}^{2j+1}(\alpha+m\beta+1)}+C.
\label{eq2.23}
\end{multline}
Using (\ref{eq2.2}) and  (\ref{eq2.3}) in Lemma \ref{lem1} yields
\begin{multline}
\int u^{\alpha}\sinh\left(u^\beta\right)du =\frac{ u^{\alpha+1}\sinh\left(u^\beta\right)}{(\alpha+1)(\alpha+\beta+1)}\sum\limits_{j=0}^{\infty}\frac{\left(\frac{ u^{2\beta}}{4}\right)^{j}}{\left(\frac{\alpha+\beta+1}{2\beta}\right)_j \left(\frac{\alpha+2\beta+1}{2\beta}\right)_j}\\ -\frac{\beta u^{\alpha+\beta+1}\cosh\left(u^\beta\right)}{(\alpha+1)(\alpha+\beta+1)}\sum\limits_{j=0}^{\infty}\frac{\left(\frac{ u^{2\beta}}{4}\right)^{j}}{\left(\frac{\alpha+2\beta+1}{2\beta}\right)_j \left(\frac{\alpha+3\beta+1}{2\beta}\right)_j}+C\\
=\frac{ u^{\alpha+1}\sinh\left(u^\beta\right)}{(\alpha+1)(\alpha+\beta+1)}\sum\limits_{j=0}^{\infty}\frac{(1)_j\left(\frac{ u^{2\beta}}{4}\right)^{j}}{\left(\frac{\alpha+\beta+1}{2\beta}\right)_j \left(\frac{\alpha+2\beta+1}{2\beta}\right)_j j!}\\ -\frac{ \beta u^{\alpha+\beta+1}\cosh\left(u^\beta\right)}{(\alpha+1)(\alpha+\beta+1)}\sum\limits_{j=0}^{\infty}\frac{(1)_j\left(\frac{ u^{2\beta}}{4}\right)^{j}}{\left(\frac{\alpha+2\beta+1}{2\beta}\right)_j \left(\frac{\alpha+3\beta+1}{2\beta}\right)_j j!}+C\\
=\frac{ u^{\alpha+1}\sinh\left(u^\beta\right)}{(\alpha+1)(\alpha+\beta+1)}\,_{1}F_{2} \Bigl(1; \frac{\alpha+\beta+1}{2\beta}, \frac{\alpha+2\beta+1}{2\beta}; \frac{u^{2\beta}}{4}\Bigr)\\ -\frac{ \beta u^{\alpha+\beta+1}\cosh\left(u^\beta\right)}{(\alpha+1)(\alpha+\beta+1)}\,_{1}F_{2} \Bigl(1; \frac{\alpha+2\beta+1}{2\beta}, \frac{\alpha+3\beta+1}{2\beta}; \frac{u^{2\beta}}{4}\Bigr)+C.
\label{eq2.24}
\end{multline}
Hence, using the fact $u^\beta=\eta x^\beta$ and rearranging terms gives (\ref{eq2.21}).
\end{proof}

\begin{theorem} For any constants $\alpha, \beta$ and $\eta$,
\begin{multline}
\frac{1}{\alpha+\beta+1}\Bigl[\cosh\left(\eta x^\beta\right)\,_{1}F_{2} \Bigl(1; \frac{\alpha+\beta+1}{2\beta}, \frac{\alpha+2\beta+1}{2\beta}; \frac{\eta^2 x^{2\beta}}{4}\Bigr)\\- \beta\eta x^{\beta}\sinh\left(\eta x^\beta\right)\,_{1}F_{2} \Bigl(1; \frac{\alpha+2\beta+1}{2\beta}, \frac{\alpha+3\beta+1}{2\beta}; \frac{\eta^2 x^{2\beta}}{4}\Bigr)\Bigr]\\
=\frac{1}{2}\left[e^{\eta x^\beta}\, _{1}F_{1} \Bigl(1;\frac{\alpha+\beta+1}{\beta}; -\eta x^{\beta}\Bigr)+e^{-\eta x^\beta}\, _{1}F_{1} \Bigl(1;\frac{\alpha+\beta+1}{\beta}; \eta x^{\beta}\Bigr)\right].
\label{eq2.25}
\end{multline}
\label{th2}
\end{theorem}

\begin{proof}
Using the hyperbolic identity $ \cosh\left(\eta x^\beta\right)= \left( e^{\eta x^\beta}+ e^{-\eta x^\beta}\right)/2$ and Proposition \ref{prp1} yields
\begin{multline}
\int x^\alpha\cosh\left(\eta x^\beta\right)dx=\frac{1}{2}\left( \int x^\alpha e^{\eta x^\beta}dx+\int x^\alpha e^{-\eta x^\beta}dx\right)\\
=\frac{x^{\alpha+1}}{2(\alpha+1)}\Bigl[e^{\eta x^\beta}\, _{1}F_{1} \Bigl(1;\frac{\alpha+\beta+1}{\beta}; -\eta x^{\beta}\Bigr)\\ +e^{-\eta x^\beta}\, _{1}F_{1} \Bigl(1;\frac{\alpha+\beta+1}{\beta}; \eta x^{\beta}\Bigr)\Bigr]+C.
\label{eq2.26}
\end{multline}
Hence, Comparing (\ref{eq2.26}) with (\ref{eq2.17}) gives (\ref{eq2.25}). 
\end{proof}

\begin{theorem} For any constants $\alpha, \beta$ and $\eta$,
\begin{multline}
\frac{1}{\alpha+\beta+1}\Bigl[\sinh\left(\eta x^\beta\right)\,_{1}F_{2} \Bigl(1; \frac{\alpha+\beta+1}{2\beta}, \frac{\alpha+2\beta+1}{2\beta}; \frac{\eta^2 x^{2\beta}}{4}\Bigr)\\- \beta\eta x^{\beta}\cosh\left(\eta x^\beta\right)\,_{1}F_{2} \Bigl(1; \frac{\alpha+2\beta+1}{2\beta}, \frac{\alpha+3\beta+1}{2\beta}; \frac{\eta^2 x^{2\beta}}{4}\Bigr)\Bigr]\\
=\frac{1}{2}\left[e^{\eta x^\beta}\, _{1}F_{1} \Bigl(1;\frac{\alpha+\beta+1}{\beta}; -\eta x^{\beta}\Bigr)-e^{-\eta x^\beta}\, _{1}F_{1} \Bigl(1;\frac{\alpha+\beta+1}{\beta}; \eta x^{\beta}\Bigr)\right].
\label{eq2.27}
\end{multline}
\label{th3}
\end{theorem}

\begin{proof}
Using the hyperbolic identity $ \sinh\left(\eta x^\beta\right)= \left( e^{\eta x^\beta}- e^{-\eta x^\beta}\right)/2$ and Proposition \ref{prp1} yields
\begin{multline}
\int x^\alpha\sinh\left(\eta x^\beta\right)dx=\frac{1}{2}\left( \int x^\alpha e^{\eta x^\beta}dx-\int x^\alpha e^{-\eta x^\beta}dx\right)\\
=\frac{x^{\alpha+1}}{2(\alpha+1)}\Bigl[e^{\eta x^\beta}\, _{1}F_{1} \Bigl(1;\frac{\alpha+\beta+1}{\beta}; -\eta x^{\beta}\Bigr)\\ -e^{-\eta x^\beta}\, _{1}F_{1} \Bigl(1;\frac{\alpha+\beta+1}{\beta}; \eta x^{\beta}\Bigr)\Bigr]+C.
\label{eq2.28}
\end{multline}
Hence, Comparing (\ref{eq2.28}) with (\ref{eq2.21}) gives (\ref{eq2.27}). 
\end{proof}

\begin{theorem} For any constants $\alpha, \beta$ and $\eta$,
\begin{multline}
e^{\eta x^\beta}\, _{1}F_{1} \Bigl(1;\frac{\alpha+\beta+1}{\beta}; -\eta x^{\beta}\Bigr)\\=\frac{1}{\alpha+\beta+1}\Bigl[\sinh\left(\eta x^\beta\right)\,_{1}F_{2} \Bigl(1; \frac{\alpha+\beta+1}{2\beta}, \frac{\alpha+2\beta+1}{2\beta}; \frac{\eta^2 x^{2\beta}}{4}\Bigr)\\- \beta\eta x^{\beta}\cosh\left(\eta x^\beta\right)\,_{1}F_{2} \Bigl(1; \frac{\alpha+2\beta+1}{2\beta}, \frac{\alpha+3\beta+1}{2\beta}; \frac{\eta^2 x^{2\beta}}{4}\Bigr)\\+\cosh\left(\eta x^\beta\right)\,_{1}F_{2} \Bigl(1; \frac{\alpha+\beta+1}{2\beta}, \frac{\alpha+2\beta+1}{2\beta}; \frac{\eta^2 x^{2\beta}}{4}\Bigr)\\- \beta\eta x^{\beta}\sinh\left(\eta x^\beta\right)\,_{1}F_{2} \Bigl(1; \frac{\alpha+2\beta+1}{2\beta}, \frac{\alpha+3\beta+1}{2\beta}; \frac{\eta^2 x^{2\beta}}{4}\Bigr)\Bigr].
\label{eq2.29}
\end{multline}
\label{th4}
\end{theorem}

\begin{proof}
The hyperbolic relation  $ e^{\eta x^\beta}=\cosh\left(\eta x^\beta\right)+\sinh\left(\eta x^\beta\right)$ and  Propositions \ref{prp2} and \ref{prp3} gives
\begin{multline}
\int x^\alpha e^{\eta x^\beta}dx= \int x^\alpha\cosh\left(\eta x^\beta\right)dx+\int x^\alpha\sinh\left(\eta x^\beta\right)dx\\
=\frac{x^{\alpha+1}}{(\alpha+1)(\alpha+\beta+1)}\Bigl[\cosh\left(\eta x^\beta\right)\,_{1}F_{2} \Bigl(1; \frac{\alpha+\beta+1}{2\beta}, \frac{\alpha+2\beta+1}{2\beta}; \frac{\eta^2 x^{2\beta}}{4}\Bigr)\\- \beta\eta x^{\beta}\sinh\left(\eta x^\beta\right)\,_{1}F_{2} \Bigl(1; \frac{\alpha+2\beta+1}{2\beta}, \frac{\alpha+3\beta+1}{2\beta}; \frac{\eta^2 x^{2\beta}}{4}\Bigr)\Bigr]\\
+\frac{x^{\alpha+1}}{(\alpha+1)(\alpha+\beta+1)}\Bigl[\sinh\left(\eta x^\beta\right)\,_{1}F_{2} \Bigl(1; \frac{\alpha+\beta+1}{2\beta}, \frac{\alpha+2\beta+1}{2\beta}; \frac{\eta^2 x^{2\beta}}{4}\Bigr)\\- \beta\eta x^{\beta}\cosh\left(\eta x^\beta\right)\,_{1}F_{2} \Bigl(1; \frac{\alpha+2\beta+1}{2\beta}, \frac{\alpha+3\beta+1}{2\beta}; \frac{\eta^2 x^{2\beta}}{4}\Bigr)\Bigr]+C.
\label{eq2.30}
\end{multline}
Hence, comparing (\ref{eq2.30}) with (\ref{eq2.10}) gives (\ref{eq2.29}).
\end{proof}

\subsection{Evaluation of non-elementary integrals of the types $\int x^\alpha \cos\left(\eta x^\beta\right)dx, \int x^\alpha \sin\left(\eta x^\beta\right)dx$}
\label{subsec:2.2}

\begin{proposition}  Let  $\eta$ and $\beta$ be nonzero constants ($\eta\ne 0, \beta\ne0$), $\alpha$ be some constant different from $-1$ ($\alpha\ne-1$) and $\alpha\ne-\beta-1$. Then,
\begin{multline}
\int x^\alpha \cos\left(\eta x^\beta\right)dx\\=\frac{x^{\alpha+1}}{(\alpha+1)(\alpha+\beta+1)}\Bigl[\cos\left(\eta x^\beta\right)\,_{1}F_{2} \Bigl(1; \frac{\alpha+\beta+1}{2\beta}, \frac{\alpha+2\beta+1}{2\beta}; -\frac{\eta^2 x^{2\beta}}{4}\Bigr)\\+\beta\eta x^{\beta}\sin\left(\eta x^\beta\right)\,_{1}F_{2} \Bigl(1; \frac{\alpha+2\beta+1}{2\beta}, \frac{\alpha+3\beta+1}{2\beta};- \frac{\eta^2 x^{2\beta}}{4}\Bigr)\Bigr]+C. 
\label{eq2.31}
\end{multline}
\label{prp4}
\end{proposition}
The proof is similar to the proof of  Proposition \ref{prp2}, so it is omitted.

\begin{proposition}  Let  $\eta$ and $\beta$ be nonzero constants ($\eta\ne 0, \beta\ne0$), $\alpha$ be some constant different from $-1$ ($\alpha\ne-1$) and $\alpha\ne-\beta-1$. Then 
\begin{multline}
\int x^\alpha \sin\left(\eta x^\beta\right)dx\\=\frac{x^{\alpha+1}}{(\alpha+1)(\alpha+\beta+1)}\Bigl[\sin\left(\eta x^\beta\right)\,_{1}F_{2} \Bigl(1; \frac{\alpha+\beta+1}{2\beta}, \frac{\alpha+2\beta+1}{2\beta}; -\frac{\eta^2 x^{2\beta}}{4}\Bigr)\\- \beta\eta x^{\beta}\cos\left(\eta x^\beta\right)\,_{1}F_{2} \Bigl(1; \frac{\alpha+2\beta+1}{2\beta}, \frac{\alpha+3\beta+1}{2\beta}; -\frac{\eta^2 x^{2\beta}}{4}\Bigr)\Bigr]+C. 
\label{eq2.32}
\end{multline}
\label{prp5}
\end{proposition}
The proof of this proposition is also omitted since it  is similar to that of  Proposition \ref{prp3}.

\begin{theorem} For any constants $\alpha, \beta$ and $\eta$,
\begin{multline}
\frac{1}{\alpha+\beta+1}\Bigl[\cos\left(\eta x^\beta\right)\,_{1}F_{2} \Bigl(1; \frac{\alpha+\beta+1}{2\beta}, \frac{\alpha+2\beta+1}{2\beta}; -\frac{\eta^2 x^{2\beta}}{4}\Bigr)\\- \beta\eta x^{\beta}\sin\left(\eta x^\beta\right)\,_{1}F_{2} \Bigl(1; \frac{\alpha+2\beta+1}{2\beta}, \frac{\alpha+3\beta+1}{2\beta}; -\frac{\eta^2 x^{2\beta}}{4}\Bigr)\Bigr]\\
=\frac{1}{2}\left[e^{i\eta x^\beta}\, _{1}F_{1} \Bigl(1;\frac{\alpha+\beta+1}{\beta}; -i\eta x^{\beta}\Bigr)+e^{-i\eta x^\beta}\, _{1}F_{1} \Bigl(1;\frac{\alpha+\beta+1}{\beta}; i\eta x^{\beta}\Bigr)\right].
\label{eq2.33}
\end{multline}
\label{th5}
\end{theorem}

\begin{proof}
Euler's identity $ \cos\left(\eta x^\beta\right)= \left( e^{i\eta x^\beta}+ e^{-i\eta x^\beta}\right)/2$ and Proposition \ref{prp1} gives
\begin{multline}
\int x^\alpha\cos\left(\eta x^\beta\right)dx= \frac{1}{2}\Bigl[\int x^\alpha e^{i\eta x^\beta}dx+\int x^\alpha e^{-i\eta x^\beta}dx\Bigr]\\
=\frac{x^{\alpha+1}}{2(\alpha+1)}\Bigl[e^{i\eta x^\beta}\, _{1}F_{1} \Bigl(1;\frac{\alpha+\beta+1}{\beta}; -i\eta x^{\beta}\Bigr)\\ +e^{-i\eta x^\beta}\, _{1}F_{1} \Bigl(1;\frac{\alpha+\beta+1}{\beta}; i\eta x^{\beta}\Bigr)\Bigr]+C.
\label{eq2.34}
\end{multline}
Hence, Comparing (\ref{eq2.34}) with (\ref{eq2.31}) gives (\ref{eq2.33}). 
\end{proof}

\begin{theorem} For any constants $\alpha, \beta$ and $\eta$,
\begin{multline}
\frac{1}{\alpha+\beta+1}\Bigl[\sin\left(\eta x^\beta\right)\,_{1}F_{2} \Bigl(1; \frac{\alpha+\beta+1}{2\beta}, \frac{\alpha+2\beta+1}{2\beta}; -\frac{\eta^2 x^{2\beta}}{4}\Bigr)\\+\beta\eta x^{\beta}\cos\left(\eta x^\beta\right)\,_{1}F_{2} \Bigl(1; \frac{\alpha+2\beta+1}{2\beta}, \frac{\alpha+3\beta+1}{2\beta}; -\frac{\eta^2 x^{2\beta}}{4}\Bigr)\Bigr]\\
=\frac{1}{2i}\left[e^{i\eta x^\beta}\, _{1}F_{1} \Bigl(1;\frac{\alpha+\beta+1}{\beta}; -i\eta x^{\beta}\Bigr)-e^{-i\eta x^\beta}\, _{1}F_{1} \Bigl(1;\frac{\alpha+\beta+1}{\beta}; i\eta x^{\beta}\Bigr)\right].
\label{eq2.35}
\end{multline}
\label{th6}
\end{theorem}

\begin{proof}
Euler's identity $ \sin\left(\eta x^\beta\right)= \left( e^{i\eta x^\beta}- e^{-i\eta x^\beta}\right)/(2i)$ and Proposition \ref{prp1} gives
\begin{multline}
\int x^\alpha\cos\left(\eta x^\beta\right)dx= \frac{1}{2}\Bigl[\int x^\alpha e^{i\eta x^\beta}dx+\int x^\alpha e^{-i\eta x^\beta}dx\Bigr]\\
=\frac{x^{\alpha+1}}{2i(\alpha+1)}\Bigl[e^{i\eta x^\beta}\, _{1}F_{1} \Bigl(1;\frac{\alpha+\beta+1}{\beta}; -i\eta x^{\beta}\Bigr)\\ -e^{-i\eta x^\beta}\, _{1}F_{1} \Bigl(1;\frac{\alpha+\beta+1}{\beta}; i\eta x^{\beta}\Bigr)\Bigr]+C.
\label{eq2.36}
\end{multline}
Hence, Comparing (\ref{eq2.36}) with (\ref{eq2.32}) gives (\ref{eq2.35}). 
\end{proof}

\begin{theorem} For any constants $\alpha, \beta$ and $\eta$,
\begin{multline}
e^{i\eta x^\beta}\, _{1}F_{1} \Bigl(1;\frac{\alpha+\beta+1}{\beta}; -i\eta x^{\beta}\Bigr)\\
=\frac{1}{\alpha+\beta+1}\Bigl[\cos\left(\eta x^\beta\right)\,_{1}F_{2} \Bigl(1; \frac{\alpha+\beta+1}{2\beta}, \frac{\alpha+2\beta+1}{2\beta}; -\frac{\eta^2 x^{2\beta}}{4}\Bigr)\\- \beta\eta x^{\beta}\sin\left(\eta x^\beta\right)\,_{1}F_{2} \Bigl(1; \frac{\alpha+2\beta+1}{2\beta}, \frac{\alpha+3\beta+1}{2\beta}; -\frac{\eta^2 x^{2\beta}}{4}\Bigr)\Bigr]\\+
\frac{i}{\alpha+\beta+1}\Bigl[\sin\left(\eta x^\beta\right)\,_{1}F_{2} \Bigl(1; \frac{\alpha+\beta+1}{2\beta}, \frac{\alpha+2\beta+1}{2\beta}; -\frac{\eta^2 x^{2\beta}}{4}\Bigr)\\+\beta\eta x^{\beta}\cos\left(\eta x^\beta\right)\,_{1}F_{2} \Bigl(1; \frac{\alpha+2\beta+1}{2\beta}, \frac{\alpha+3\beta+1}{2\beta}; -\frac{\eta^2 x^{2\beta}}{4}\Bigr)\Bigr].
\label{eq2.37}
\end{multline}
\label{th7}
\end{theorem}

\begin{proof}
Using the relation  $ e^{i\eta x^\beta}=\cos\left(\eta x^\beta\right)+i\sin\left(\eta x^\beta\right)$ and  Propositions \ref{prp4} and \ref{prp5} yields
\begin{multline}
\int x^\alpha e^{i\eta x^\beta}dx= \int x^\alpha\cos\left(\eta x^\beta\right)dx+i\int x^\alpha\sin\left(\eta x^\beta\right)dx\\=\frac{x^{\alpha+1}}{(\alpha+1)(\alpha+\beta+1)}\Bigl[\cos\left(\eta x^\beta\right)\,_{1}F_{2} \Bigl(1; \frac{\alpha+\beta+1}{2\beta}, \frac{\alpha+2\beta+1}{2\beta}; -\frac{\eta^2 x^{2\beta}}{4}\Bigr)\\+\beta\eta x^{\beta}\sin\left(\eta x^\beta\right)\,_{1}F_{2} \Bigl(1; \frac{\alpha+2\beta+1}{2\beta}, \frac{\alpha+3\beta+1}{2\beta};- \frac{\eta^2 x^{2\beta}}{4}\Bigr)\Bigr]\\+i \frac{x^{\alpha+1}}{(\alpha+1)(\alpha+\beta+1)}\Bigl[\sin\left(\eta x^\beta\right)\,_{1}F_{2} \Bigl(1; \frac{\alpha+\beta+1}{2\beta}, \frac{\alpha+2\beta+1}{2\beta}; -\frac{\eta^2 x^{2\beta}}{4}\Bigr)\\- \beta\eta x^{\beta}\cos\left(\eta x^\beta\right)\,_{1}F_{2} \Bigl(1; \frac{\alpha+2\beta+1}{2\beta}, \frac{\alpha+3\beta+1}{2\beta}; -\frac{\eta^2 x^{2\beta}}{4}\Bigr)\Bigr]+C.
\label{eq2.38}
\end{multline}
Hence, comparing (\ref{eq2.38}) with (\ref{eq2.10}) (with $\eta$ replaced by $i\eta$ ) gives (\ref{eq2.37}).
\end{proof}
\section{New probability measures that generalize the gamma-type and  Gaussian-type distributions}
\label{sec:3}
In this section, Theorem \ref{th1} is used to generalize the gamma-type ($\chi^2$ distribution, inverse gamma distribution)  distribution and Gaussian-type distributions.

\subsection{Generalization of the  gamma-type  distributions}
\label{sybsec:3.1}
Define a probability measure $\mu$ in terms of the Lebesgue measure $dx$ as \cite{B} 
\begin{equation}
d\mu=\mu(dx)= A\,g(x;\alpha,\eta,\beta) dx=f_X(x;\alpha,\eta,\beta) dx, \,\, x\in [0,+\infty),
\label{eq3.1}
\end{equation}
where $f_X(x;\alpha,\eta,\beta)$ is the probability density function (p.d.f.) of some random variable $X$, 
\begin{equation}
g(x;\alpha,\eta,\beta)=x^\alpha e^{-\eta x^\beta}, \, \alpha \ne -1,\,\beta\ne0,\, \alpha>-\beta-1,
\label{eq3.2}
\end{equation}
and $A$ is a normalized constant which can be obtained using formula (\ref{eq2.14}) in Theorem \ref{th1}. 

After normalization, it is found that the p.d.f. of $X$ is given by 
\begin{equation}
f_X(x;\alpha,\eta,\beta)=\frac{(\alpha+1)\,\eta^{{(\alpha+1)}/{\beta}}}{\Gamma \left({(\alpha+\beta+1)}/{\beta}\right)}x^\alpha e^{-\eta x^\beta}, \, \alpha \ne -1 ,\,\beta\ne0,\, \alpha>-\beta-1.
\label{eq3.3}
\end{equation}
The distribution function of the random variable $X$  can be obtained using Proposition \ref{prp1} and is  given by
\begin{multline}
F_X(x;\alpha,\eta,\beta)=\mu\{[0,x)\}=\int\limits_0^x f_X(u;\alpha,\eta,\beta) du\\=\frac{(\alpha+1)\,\eta^{{(\alpha+1)}/{\beta}}}{\Gamma \left({(\alpha+\beta+1)}/{\beta}\right)}x^{\alpha+1}e^{-\eta x^\beta}\, _{1}F_{1} \Bigl(1;\frac{\alpha+\beta+1}{\beta}; \eta x^{\beta}\Bigr).
\label{eq3.4}
\end{multline}
The $\mbox{n}^{th}$ moments ($M(X^n)$) can, as well, be evaluated using formula (\ref{eq2.14}) in Theorem \ref{th1} to obtain
\begin{align}
M(X^n)&=\int\limits_0^{+\infty}  x^n f_X(x;\alpha,\eta,\beta) dx
\nonumber \\ &=\frac{(\alpha+1)\,\eta^{{(\alpha+1)}/{\beta}}}{\Gamma \left({(\alpha+\beta+1)}/{\beta}\right)}\int\limits_0^{+\infty}  x^{\alpha+n} e^{-\eta x^\beta} dx
\nonumber \\ &=\frac{(\alpha+1)\,\eta^{{(\alpha+1)}/{\beta}}}{\Gamma \left({(\alpha+\beta+1)}/{\beta}\right)}\frac{\Gamma \left({(\alpha+\beta+n+1)}/{\beta}\right)}{(\alpha+n+1)\,\eta^{{(\alpha+n+1)}/{\beta}}}=\frac{\Gamma \left({(\alpha+n+1)}/{\beta}\right)}{\eta^{\alpha/\beta}\Gamma \left({(\alpha+1)}/{\beta}\right)}.
\label{eq3.5}
\end{align}

These results are summarized in the following theorem.

\begin{theorem}
Let $X$  be a random variable with the generalized gamma-type p.d.f.
\begin{equation}
f_X(x;\alpha,\eta,\beta)=\frac{(\alpha+1)\,\eta^{{(\alpha+1)}/{\beta}}}{\Gamma \left({(\alpha+\beta+1)}/{\beta}\right)}x^\alpha e^{-\eta x^\beta}, \,x\in\mathbb{R}^+,\, \alpha \ne -1, \,\beta\ne0,\, \alpha>-\beta-1.
\label{eq3.6}
\end{equation}
 Then, the distribution function $F_X(x;\alpha,\eta,\beta)$ of the random variable $X$ is given by
\begin{equation}
F_X(x;\alpha,\eta,\beta)=\frac{(\alpha+1)\,\eta^{{(\alpha+1)}/{\beta}}}{\Gamma \left({(\alpha+\beta+1)}/{\beta}\right)}x^{\alpha+1}e^{-\eta x^\beta}\, _{1}F_{1} \Bigl(1;\frac{\alpha+\beta+1}{\beta}; \eta x^{\beta}\Bigr), 
\label{eq3.7}
\end{equation}
and the $\mbox{n}^{th}$ moments $M(X^n)$ of $X$ are given by 
\begin{align}
M(X^n)=\int\limits_0^{+\infty}  x^n f_X(x;\alpha,\eta,\beta) dx
=\frac{\Gamma \left({(\alpha+n+1)}/{\beta}\right)}{\eta^{\alpha/\beta}\Gamma \left({(\alpha+1)}/{\beta}\right)}.
\label{eq3.8}
\end{align}
\label{th8}
\end{theorem}
 
If $Y$ is some gamma distribution random variable, then the random variable $X=1/Y$  is said to be an inverse gamma distribution random variable. The inverse gamma distribution find applications in wireless communications, see for example \cite{C,W}. Its distribution function may be evaluated using Theorem \ref{th8}.

\begin{corollary} Let $X$  be a random variable with the inverse gamma distribution, $X\sim IG(\theta,\eta)$. Then, the distribution function $F_X(x;\theta,\eta)$ is given by 
\begin{equation}
F_X(x;\theta,\eta)=\frac{-\eta^{\theta}}{\Gamma(\theta+1)} x^{-\theta}e^{-\eta/x}\, _{1}F_{1} \Bigl(1;\theta+1; \eta/ x\Bigr), \, x>0, \theta>0, \eta>0,
\label{eq3.9}
\end{equation}
while the $\mbox{n}^{th}$ moments $M(X^n)$ are given by 
\begin{align}
M(X^n)=\frac{\eta^{n}\Gamma(\theta-n)}{\Gamma(\theta+1)}, n>\theta.
\label{eq3.10}
\end{align}
\label{co1}
\end{corollary} 

\begin{proof}
Setting $\alpha=-(\theta+1), \beta=-1$, and using the fundamental theorem of calculus $f_X(x)=\frac{d F_X}{dx}=\frac{d}{dx}\int_0^x f_X(u) du$ and Proposition \ref{prp1} gives the p.d.f.
\begin{equation}
f_X(x;\alpha,\eta,\beta)=\frac{\eta^{\theta}}{\Gamma(\theta)}x^{-(\theta+1)} e^{-\eta x^{-1}}, \, x>0, \theta>0, \eta>0,
\label{eq3.11}
\end{equation}
which is the p.d.f of the inverse gamma distribution. The $\mbox{n}^{th}$  moments $M(X^n)$ of $X$ are obtained by setting $\alpha=-(\theta+1)$ and $\beta=-1$ in (\ref{eq3.8}).
\end{proof}

\subsection{Generalization of Gaussian-type  distributions}
\label{sybsec:3.12}
Consider a probability measure $\mu$ in terms of Lebesgue measure $dx$ given by \cite{B}
\begin{equation}
d\mu=\mu(dx)= A\,g(x;\alpha,\eta,\beta) dx=f_X(x;\alpha,\eta,\beta) dx, \,\, x\in\mathbb{R},
\label{eq3.12}
\end{equation}
where, as before, $f_X(x;\alpha,\eta,\beta)$ is the p.d.f. of some random variable $X$, 
\begin{equation}
g(x;\alpha,\eta,\beta)=x^\alpha e^{-\eta x^\beta}, \, \alpha\ne-1 ,\, \beta\ne0,\alpha>-\beta-1,
\label{eq3.13}
\end{equation}
is an even function of the variable $x$, and $A$ is a normalized constant which can be obtained using formula (\ref{eq2.15}) in Theorem \ref{th1}. 

After normalization, the  p.d.f. of $f_X$ is found to be 
\begin{equation}
f_X(x;\alpha,\eta,\beta)=\frac{(\alpha+1)\,\eta^{{(\alpha+1)}/{\beta}}}{2\Gamma \left({(\alpha+\beta+1)}/{\beta}\right)}x^\alpha e^{-\eta x^\beta},\, x\in\mathbb{R}, \,\alpha\ne-1 ,\, \beta\ne0,\alpha>-\beta-1.
\label{eq3.14}
\end{equation}
It is important to note that $f_X$ in this case is even, and so, a factor of 2 has to appear in the denominator.
The distribution function $F_X$  can also be obtained using Proposition \ref{prp1} and is thus given by 
\begin{multline}
F_X(x;\alpha,\eta,\beta)=\mu\{(-\infty,x)\}=\int\limits_{-\infty}^x f_X(u;\alpha,\eta,\beta) du\\=\frac{1}{2}\left[1-\frac{(\alpha+1)\,\eta^{{(\alpha+1)}/{\beta}}}{\Gamma \left({(\alpha+\beta+1)}/{\beta}\right)}x^{\alpha+1}e^{-\eta x^\beta}\, _{1}F_{1} \Bigl(1;\frac{\alpha+\beta+1}{\beta}; \eta x^{\beta}\Bigr)\right].
\label{eq3.15}
\end{multline}
The moment ($M(X^n)$) can be evaluated using formula (\ref{eq2.15}) in Theorem \ref{th1} to obtain
\begin{align}
M(X^n)&=\int\limits_{-\infty}^{+\infty}  x^n f_X(x;\alpha,\eta,\beta) dx
\nonumber \\ &=\frac{(\alpha+1)\,\eta^{{(\alpha+1)}/{\beta}}}{2\Gamma \left({(\alpha+\beta+1)}/{\beta}\right)}\int\limits_{-\infty}^{+\infty}  x^{\alpha+n} e^{-\eta x^\beta} dx
\nonumber \\ &=\begin{cases} \frac{\Gamma \left({(\alpha+n+1)}/{\beta}\right)}{\eta^{\alpha/\beta}\Gamma \left({(\alpha+1)}/{\beta}\right)} , & \text{if}\,\, n \,\text{is even}.  \\
0, & \text{if}\,\, n \,\text{is odd}. \end{cases}
\label{eq3.16}
\end{align}

These results are summarized in the following theorem.

\begin{theorem}
Let $X$  be a random variable with an even p.d.f. of the form
\begin{equation}
f_X(x;\alpha,\eta,\beta)=\frac{(\alpha+1)\,\eta^{{(\alpha+1)}/{\beta}}}{2\Gamma \left({(\alpha+\beta+1)}/{\beta}\right)}x^\alpha e^{-\eta x^\beta}, \,x\in\mathbb{R} ,\, \alpha\ne-1 ,\, \beta\ne0,\alpha>-\beta-1.
\label{eq3.17}
\end{equation}
 Then, the distribution function $F_X(x;\alpha,\eta,\beta)$ of the random variable $X$ is given by
\begin{equation}
F_X(x;\alpha,\eta,\beta)=\frac{1}{2}\left[1-\frac{(\alpha+1)\,\eta^{{(\alpha+1)}/{\beta}}}{\Gamma \left({(\alpha+\beta+1)}/{\beta}\right)}x^{\alpha+1}e^{-\eta x^\beta}\, _{1}F_{1} \Bigl(1;\frac{\alpha+\beta+1}{\beta}; \eta x^{\beta}\Bigr)\right].
\label{eq3.18}
\end{equation}
and the $\mbox{n}^{th}$ moments $M(X^n)$ of $X$ are given by  
\begin{align}
M(X^n)=\int\limits_{-\infty}^{+\infty}  x^n f_X(x;\alpha,\eta,\beta) dx
=\begin{cases} \frac{\Gamma \left({(\alpha+n+1)}/{\beta}\right)}{\eta^{\alpha/\beta}\Gamma \left({(\alpha+1)}/{\beta}\right)} , & \text{if}\,\, n \,\text{is even}.  \\
0, & \text{if}\,\, n \,\text{is odd}. \end{cases}
\label{eq3.19}
\end{align}
\label{th9}
\end{theorem}
For example, setting $\alpha=0, \beta=2$ and $\eta=1/2$ yields $f_X(x)=(1/\sqrt{2\pi})e^{-x^2/2}$,  and the mean of $X$ is $E\/X=M(X^1)=0$ while the variance is $E\/X^2=M(X^2)=1$. So $X\sim N(0,1)$ distribution as expected.

More general results can be achieved by introducing two additional parameters.
\begin{theorem}
Let $X$  be a random variable with an even p.d.f.of the form
\begin{multline}
f_X(x;\alpha,\eta,\beta,\theta,\sigma)=\frac{1}{2\sigma}\frac{(\alpha+1)\,\eta^{{(\alpha+1)}/{\beta}}}{\Gamma \left({(\alpha+\beta+1)}/{\beta}\right)}
\\ \times\left(\frac{x-\theta}{\sigma}\right)^\alpha \exp\left(-\eta\left(\frac{x-\theta}{\sigma}\right)^\beta\right),x\in\mathbb{R}, \, \theta\in\mathbb{R} ,\, \alpha\ne-1 ,\, \beta\ne0,\alpha>-\beta-1,\, \sigma>0.
\label{eq3.20}
\end{multline}
 Then, the distribution function $F_X(x;\alpha,\eta,\beta,\theta,\sigma)$ of the random variable $X$ is given by
\begin{multline}
F_X(x;\alpha,\eta,\beta,\theta,\sigma)=\frac{1}{2}\Bigl[1-\frac{1}{\sigma}\frac{(\alpha+1)\,\eta^{{(\alpha+1)}/{\beta}}}{\Gamma \left({(\alpha+\beta+1)}/{\beta}\right)}\left(\frac{x-\theta}{\sigma}\right)^{\alpha+1}\\\exp\left(-\eta\left(\frac{x-\theta}{\sigma}\right)^\beta\right)\, _{1}F_{1} \left(1;\frac{\alpha+\beta+1}{\beta}; \eta\left(\frac{x-\theta}{\sigma}\right)^{\beta}\right)\Bigr],
\label{eq3.21}
\end{multline}
and the moments $M(X^n)$ of $X$ are given by
\begin{multline}
M(X^n)=\int\limits_{-\infty}^{+\infty}  x^n f_X(x;\alpha,\eta,\beta) dx\\
= \frac{\theta^n}{\Gamma \left({(\alpha+1)}/{\beta}\right)}\sum\limits_{l=0}^{n}\Gamma \left({(\alpha+2l+1)}/{\beta}\right)C_{2l}^n\left(\frac{\sigma}{\theta \, \eta^{1/\beta}}\right)^{2l}, \, 2l\le n\, \mbox{and}\,(2l)\in\mathbb{N},
\label{eq3.22}
\end{multline}
where $C_{2l}^n=n!/((n-2l)! (2l)!)$.

Thus, the mean and the variance of $X$ are respectively given by 
\begin{equation}
E X= M(X^1)=\theta\,\, \mbox{and}\,\,  \text{var X}=EX^2-(E X)^2=\frac{\sigma^2}{\eta^{2/\beta}}\frac{\Gamma \left({(\alpha+3)}/{\beta}\right)}{\Gamma \left({(\alpha+1)}/{\beta}\right)}.
\label{eq3.23}
\end{equation}
\label{th10}
\end{theorem}
Formula (\ref{eq3.22}) is obtained by making the substitution $u=(x-\theta)/\sigma$, and by applying the binomial theorem and Theorem \ref{th1}.

A generalized Gaussian-type distribution may be derived by setting $\alpha=0$ in Theorem \ref{th10}.

\begin{corollary}
Let $X$  be a random variable with the generalized Gaussian-type distribution p.d.f. of the form
\begin{equation}
f_X(x;\eta,\beta,\theta,\sigma)=\frac{\beta}{2\sigma}\frac{\eta^{1/\beta}}{\Gamma \left(1/\beta\right)}
\exp\left(-\eta\left(\frac{x-\theta}{\sigma}\right)^\beta\right),x\in\mathbb{R}, \, \theta\in\mathbb{R} ,\, \beta>0, \sigma>0,
\label{eq3.24}
\end{equation}
and where $\beta$ is even.  Then, the distribution function $F_X(x;\alpha,\eta,\beta,\theta,\sigma)$ of the random variable $X$ is given by
\begin{multline}
F_X(x;\eta,\beta,\theta,\sigma)=\frac{1}{2}\Bigl[1-\frac{\beta}{\sigma}\frac{\eta^{1/\beta}}{\Gamma \left(1/\beta\right)}\times \\ \exp\left(-\eta\left(\frac{x-\theta}{\sigma}\right)^\beta\right)\, _{1}F_{1} \left(1;\frac{\beta+1}{\beta}; \eta\left(\frac{x-\theta}{\sigma}\right)^{\beta}\right)\Bigr],
\label{eq3.25}
\end{multline}
and the $\mbox{n}^{th}$  moment $M(X^n)$ of $X$ are given by 
\begin{multline}
M(X^n)=\int\limits_{-\infty}^{+\infty}  x^n f_X(x;\eta,\beta,\theta,\sigma) dx\\
= \frac{\theta^n}{\Gamma \left({1}/{\beta}\right)}\sum\limits_{l=0}^{n}\Gamma \left({(2l+1)}/{\beta}\right)C_{2l}^n\left(\frac{\sigma}{\theta \, \eta^{1/\beta}}\right)^{2l}, \, 2l\le n\, \mbox{and}\,(2l)\in\mathbb{N},
\label{eq3.26}
\end{multline}
where, as before, $C_{2l}^n=n!/((n-2l)! (2l)!)$.
Thus, the mean and the variance of $X$ are respectively given by 
\begin{equation}
E X= M(X^1)=\theta\,\, \mbox{and}\,\,  \text{var X}=EX^2-(E X)^2=\frac{\sigma^2}{\eta^{2/\beta}}\frac{\Gamma \left(3/\beta\right)}{\Gamma \left(1/\beta\right)}.
\label{eq3.27}
\end{equation}
Moreover, if $\eta\ge1/2$ and $\beta>2$, and since ${\Gamma \left(3/\beta\right)}<{\Gamma \left(1/\beta\right)}$, then the variance of $X$ ($\text{var}\,X$) does satisfy
\begin{equation}
\text{var}\,X=EX^2-(E X)^2=\frac{\sigma^2}{\eta^{2/\beta}}\frac{\Gamma \left(3/\beta\right)}{\Gamma \left(1/\beta\right)}\le \sigma^2,
\label{eq3.27}
\end{equation}
where $\sigma^2$ is the variance of the Gaussian random variable.
\label{co2}
\end{corollary}

A formula for the $\mbox{n}^{th}$  moments of the Gaussian distribution can now be obtained by setting $\beta=2, \eta=1/2$ and $\alpha=0$ in (\ref{eq3.26}).

\begin{corollary}
Let $X$  be a Gaussian random variable. Its $\mbox{n}^{th}$ moments $M(X^n)$ are thus given by the formula
\begin{equation}
M(X^n)
= \frac{\theta^n}{\sqrt{\pi}}\sum\limits_{l=0}^{n}\Gamma \left(l+1/2\right)C_{2l}^n\left(\frac{2\sigma^2}{\theta^2}\right)^{l}, \, 2l\le n\, \mbox{and}\,(2l)\in\mathbb{N},
\label{eq3.28}
\end{equation}
where $\theta\in \mathbb{R}$ is the mean of the Gaussian random variable and $\sigma^2>0$ its variance, and as before, $C_{2l}^n=n!/((n-2l)! (2l)!)$.
\label{co3}
\end{corollary}


%

\section{Concluding remarks and discussion}\label{sec:4}
Formulas for the integrals $\int x^\alpha e^{\eta x^\beta}dx, \int x^\alpha \cosh\left(\eta x^\beta\right)dx, \int x^\alpha \sinh\left(\eta x^\beta\right)dx, \\\int x^\alpha \cos\left(\eta x^\beta\right)dx$ and $\int x^\alpha \sin\left(\eta x^\beta\right)dx $
where $\alpha, \eta$ and $\beta$ are real or complex constants were obtained in terms of the confluent hypergeometric function $_1F_1$ and  the  hypergeometric function $_1F_2$ in section \ref{sec:2} (Propositions \ref{prp1}, \ref{prp2}, \ref{prp3}, \ref{prp4} and \ref{prp5}), and using hyperbolic and Euler identities, some identities involving confluent hypergeometric function $_1F_1$ and hypergeometric function $_1F_2$ were also obtained  in section \ref{sec:2} (Theorems \ref{th2}-\ref{th7}). Having evaluated the integrals $\int_{-\infty}^{+\infty} x^\alpha e^{-\eta x^\beta}dx, \eta>0$ in Theorem \ref{th1}, probability measures that generalize the gamma-type and Gaussian distributions were constructed. Their distributions were written in terms of the confluent hypergeometric function $_1F_1$, and formulas for the $\mbox{n}^{th}$ moments were obtained as well in section \ref{sec:3}  (Theorems \ref{th8}-\ref{th10} and Corollaries \ref{co2}-\ref{co3}). The results obtained in this paper may, for example, may be used to construct  better statistical tests than those already know (e.g. $\chi^2$ statistical tests and tests obtained  based on the normal distribution).

\end{document}